\documentclass[12pt]{amsart}
\usepackage{hyperref}
\usepackage{amsmath,amsthm,amsfonts,amssymb,amscd}
\theoremstyle{plain}
\usepackage[svgnames]{xcolor}
\usepackage{mathrsfs}
\usepackage{eepic}
\usepackage[centering]{geometry}

\newtheorem{theorem}{\bf Theorem}[section]
\newtheorem{lem}[theorem]{\bf Lemma}
\newtheorem{defi}[theorem]{\bf Definition}
\newtheorem{rem}[theorem]{\bf Remark}
\newtheorem{cor}[theorem]{\bf Corollary}
\newtheorem{ex}[theorem]{\bf Example}
\newtheorem{conv}[theorem]{\bf Convention}
\newtheorem*{conj*}{\bf Conjecture}

\def\rank{\ensuremath{\mathrm{rank}}}

\def\diam{\ensuremath{\mathrm{diam}}}

\begin{document}
\keywords{Diameter of a group, rank of a group, non-abelian simple groups}
\title{Diameter of a direct power of alternating groups}
\author{ Azizollah Azad}
\address {Department of Mathematics, Faculty of Sciences Arak University, Arak, Iran.}
\email{a-azad@araku.ac.ir}
\author{ Nasim Karimi}
\address{Instituto de Matem\'atica e estat\'istica, Universidade do Estado de  Rio de janeiro, Rio de Janeiro,Brasil}
\email{nasim@ime.uerj.br}
\date{\today}
\maketitle
\keywords 
\begin{abstract} 
So far, it has been proven that if $G$ is an abelian group , then  the diameter of $G^n$ with respect to any generating set is  $O(n)$; and if $G$ is nilpotent, symmetric or dihedral, then there exists a generating set of minimum size, for which the diameter of $G^n$ is $O(n)$ \cite{Karimi:2017}. In \cite{Dona:2022} it has been proven that if $G$ is a non-abelian simple group, then the diameter of $G^n$ with respect to any generating set is $O(n^3)$. In this paper we estimate the diameter of direct power of alternating groups $A_n$ for $n \geq 4$,  i.e. a class of non-abelian simple groups. We show that there exist a generating set of minimum size for  $A_4^n$, for which the diameter of $A_4^n$ is $O(n)$. 
For $n \geq 5$, we show that there exists a generating set of minimum size for  $A_n^2$, for which the diameter of $A_n^2$ is at most $O(ne^{(c+1) (\log \,n)^4 \log \log n})$
, for an absolute constant $c >0$. Finally for $ 1\leq n \leq 8 $, we provide generating sets of size two for $A_5^n$ and we show that the  diameter of $A_5^n$ with respect to those generating sets is $O(n)$. These results are more pieces of evidence for a conjecture which has been presented in \cite{Karimithesis:2015} in 2015.  
\end{abstract}

%\begin{abstract}  In \cite{Karimi:2017} is presented two conjectures concerning the diameter of a direct power of a finite group. The first conjecture states that 
%the diameter of $G^n$ with respect to any generating set is at most $n(|G|-\rank(G))$; and  
% the second one states that there exists a generating set $A$, of minimum size, for $G^n$ such that
%the diameter of $G^n$ with respect to $A$ is at most  $n(|G|-\rank(G))$.
% The first conjecture is proved for abelain groups and the second one is proved for nilpotent groups, symmetric groups and dihedral groups. In this paper we show that the second conjecture is satisfied the alternating group $A_4$ for $n\geq 1$ and alternating group  $A_5$ for $1\leq n \leq 8$.  
% \end{abstract}
%_______________________________________________________________________________________________________________________________

\section{Introduction}
Let $G$ be a finite group with a generating set $A$. By diameter of $G$ with respect to $A$ we mean the maximum over $g \in G$ of the length of the shortest word in $A$ expressing $g$. Finding a bound for the diameter of a finite group is an important area of research in finite group theory.  We mention the most important conjecture in this area, known as the Babai's conjecture \cite{Babai&Seress:1988}: every non-abelain finite simple group G has diameter less than or equal to $log^k |G|$, where $k$ is an absolute constant. The conjecture is still open, despite great progress towards a solution both for alternating groups and for groups of Lie type.

  Producing a bound for the diameter of a direct product of simple groups, depending on the diameter of their factors, have been used more than once for proving the Babai's conjecture. In \cite{Babai&Seress:1992}, it has shown that for $G=T_1\times T_2 \times \cdots T_n$, in which $T_i$ are non-abelain simple groups, $\diam(G) \leq 20 n^3 h^2 $, such that $h$ is the maximum diameter of $T_i$'s.  In \cite{Helfgott:2019}, this bound improved to be a bound depending on $h$ instead of $h^2$, when the factors are alternating groups; and then in \cite{Dona:2022}, it is generalized  for every non-abelain simple groups. So far, all the upper bounds presented, are depending on $n^3$. This paper is organized as follows:
  
  In section \ref{sec:1}, we find generating sets of minimum size for $A_4^n$  for $n\geq 1$, and we show that the diameter of $A_4^n$, with respect to those generating sets is $O(n)$ for $n\geq 2$.

In section \ref{sec:2}, for $n\geq 5$, we find generating sets of size two for $A_n^2$ , for which the diameter of $A_n^2$ is $O(n e^{(c+1)(\log n)^4 \log \log n})$, for an absolute constant $c >0$. 

In section \ref{sec:3}, we show that there exist generating sets of size two for $A_5^n\, (1 \leq n\leq 8) $ for which, the diameter of $A_5^n$ is at most $n \,(|A_5|-\rank(A_5))=58\,n$.

\section{Preliminaries}
 Throughout the paper all groups are considered to be finite . The subset $A\subseteq G$ is a generating set of $G$, if every element of $G$ can be expressed as a sequence of elements in $A$.\footnote{Usually $A\subseteq G$ is considered to be a generating set, if every element of $G$ can be expressed as a sequence of elements in $A\cup A^{-1}$. When $G$ is finite the definitions coincide.}  By the rank of $G$, denoted by $\rank(G)$,  we mean the cardinality of any of the smallest generating sets of $G$. 
 By the length of a non identity element $g\in G$, with respect to $A$, 
 we mean the minimum length of a sequence expressing $g$ in terms of elements in $A$.
 Denote this parameter by $l_A{g}$. 
 
 %Similarly we define the symmetric length of  a non identity element $g\in G$, with respect to $A$,  to be the minimum length of a sequence expressing $g$ in terms of elements in $A \cup A^{-1}$. Denote this parameter by $l^{s}_A{g}$.
 
 \begin{conv}
 We consider the length of identity to be zero, i.e. $l_A(1)=0$ for every generating set $A$.  
 \end{conv}
 
 \begin{defi}
 Let $G$ be a finite group with generating set $A$. By the  \emph{diameter} of $G$ with respect to $A$ we mean  
  $$\diam(G,A):=\max \{l_A(g): g \in G \}.$$ 
% And by the symmetric \emph{diameter} of $G$ with respect to $A$ we mean
%  $$\diam^{s}(G,A):=\max \{l^{s}_A(g): g \in G \}.$$  

 \end{defi} 

% \begin{notation}
% Let $G$ be a finite group with a generating set $A$. Let $S$ be a subset of $G$. Denote by $Ml_A(S)$ the maximum of $l_A(s)$ over all $s$ in $S$ . Note that $Ml_A(G)=\diam(G,A).$
% \end{notation}
 
%\begin{notation}
%Denote by $D(G)$ the maximum diameter over all generating sets of $G$.
%\end{notation} 

The next definition introduces a generating set (let us call it canonical) for any direct power $G^n$ with respect to a generating set of $G$. 
\begin{defi}\label{canonical-generating-set} 
Let $G$ be a finite group with a generating set $A$. By the \emph{canonical generating set} of $G^n$ with respect to $A$, we mean the set   
$$C^n(A):=\{(1,\ldots,\overbrace{a}^{i\,\mathrm{th}},\ldots,1) : i \in \{1,2,\ldots, n \}, a \in A \}.$$
\end{defi} 

\begin{rem}\label{property} 

 If $G$ be a group with the property that $\rank(G^n) = n \, \rank(G)$ then the canonical generating set of $G^n$ is a generating set of minimum size and the diametr of $G^n$ with respect to  $C^n(A)$ is at most  $O(n)$ \cite{Karimi:2017}. Note that the alternating groups $A_n$ for $n\geq 4$,  dose not have the property that $\rank(G^n) = n \, \rank(G)$.    
 
 %It means that $G$  satisfies the weak conjecture . 
\end{rem}

We explain the following easy fact as a remark. 
\begin{rem}\label{simple}
Let $(g_1,g_2,\ldots,g_n) \in G^n=\langle A \rangle $. Since
$ (g_1,g_2,\ldots,g_n)$ is a product of $n$ elements of the form  $(1,\ldots,g_i,\ldots,1),$ then we have 
\begin{equation}\label{length1} 
l_A(g_1,g_2,\ldots,g_n)\leq \sum_{i=1}^n l_A(1,\ldots,g_i,\ldots,1). 
\end{equation}
\end{rem}
\begin{defi}
By an \emph{$n$-basis} of a group $G$ we mean any ordered set of $n$ elements $x_1,x_2,\ldots,x_n$ of $G$ which generates $G$. Furthermore, two $n$-bases $x_1,x_2,\ldots,x_n$ and  $y_1,y_2,\ldots,y_n$ of $G$ will be called \emph{equivalent} if there exists an automorphism $\theta$ of $G$ which transforms one into the other:
$$ x_i \theta = y_i,$$
for each $i=1,2,\ldots,n$. Otherwise the two bases will be called \emph{non-equivalent}. 
\end{defi}

In general, we have the following lemma for the rank of a direct power of a finite group $G$:
\begin{lem}\label{lem:rank-inequality} {\rm \cite{Wiegold:1974}}
Let $G$ be a finite group and $k$ be a positive integer. The following inequalities hold:
\begin{equation}\label{rank-inequality}
k\, \rank (G/G') \leq \rank (G^k) \leq k\, \rank(G),
\end{equation}
where $G'$ is the commutator subgroup of $G$.
\end{lem}

\begin{defi}
A group is said to be \emph{perfect} if it equals its own commutator subgroup; otherwise it is called \emph{imperfect}.
\end{defi}

\begin{rem}\label{rem:rank}
By Lemma \ref{lem:rank-inequality},  if $G$ is a perfect group, then the lower bound in the inequality \eqref{rank-inequality} is zero, hence the first inequality in  lemma \ref{lem:rank-inequality} is trivial. If $G$ is imperfect, then the first inequality in lemma \ref{lem:rank-inequality} gives a lower bound, depending on $k$, for the rank of $G^k$. 
\end{rem}

Since alternating group $A_4$ is imperfect; and for $n\geq 5$ alternating groups $A_n$ are perfect, then by  Remark \ref{rem:rank}, we need to verify them in the separate sections. 

%_______________________________________________________________________________________________________________
 %_________________________________________Evidence________________________________________________________________

\section{The diameter of a direct power of alternating group $A_4$ }\label{sec:1}
%Note that for $n \leq 3$ alternating group $A_n$ is abelain then satisfies the strong and consequently the weak  conjecture \cite{Karimi:2017}. Then we consider the alternating groups $A_n$ for $n\geq 4$. 
%For proving the weak conjecture, the first thing which we need is to find a generating set of minimum size for $G^n$. This generating set may be the canonical generating set,  when $G$ has the property $\
%\rank (G^n) = n \, \rank(G)$ (see Remark \ref{property}).  But when $G$ dose not have the property  $ \rank(G^n) = n \, \rank(G)$, finding such generating set it is  not obvious.  Alternating group $A_n$  for $n \geq 4$ dose not have the property $\rank(G^n) = n \, \rank(G)$ . Then we have to use another techniques to  find a generating set of minimum size for $A_4^n$ and $A_5^n$. 

%Finally we show that $A_4$ satisfies the weak conjecture for $n\geq 1 $ and $A_5$ satisfies the weak conjecture for $1\leq n \leq 19$.    

We use the following lemma for finding a generating set of minimum size for a direct power of alternating group $A_4$.

\begin{lem}\label{rank-G^n}
Let $G$ be a finite imperfect group. If $G$ is generated by $k$ elements of mutually coprime orders, then 
$\rank(G^n)=n,$ for $n \geq k$.
\end{lem}
\begin{proof}
Because $G$ is not perfect, it follows from Lemma \ref{lem:rank-inequality} that $\rank(G^n) \geq n$. Suppose
$A=\{a_1,a_2,\ldots,a_k\}$ is a generating set of $G$ such that the $a_i$'s are of mutually coprime orders. Let $n \geq k$. We construct a generating set of size $n$ for $G^n$. For $1\leq i \leq n$, define the elements $g_i \in G^n$ as follows:
\begin{align*}  
&g_i=(1,\ldots, \overbrace{a_1}^{i\,\mathrm{th}},a_2,\ldots,a_k,\ldots,1) & \mbox{for}~ 1 \leq i \leq n-k+1,\\ 
&g_i=(a_{n-i+2},a_{n-i+3},\ldots,a_k,1,\ldots, 1,\overbrace{a_1}^{i\,\mathrm{th}},\ldots,a_{n-i+1})& \mbox{for}~ n-k+2 \leq i \leq n.
\end{align*}
We prove that $C=\{g_1,g_2,\ldots,g_n\}$ is a generating set of $G^n$. If we show that $C$ generates $C^n(A)$, then we are done. Choose an arbitrary element $(1,\ldots,a_i,\ldots,1) \in C^n(A)$. Since the $a_i$'s are of mutually coprime orders, there exists a positive integer $\ell$ such that 
\begin{equation*}(1,\ldots,a_i,\ldots,1) =
(1,\ldots, a_1,\ldots,a_i,\ldots,a_k,\ldots,1)^{\ell}.
\end{equation*}
This yields the desired conclusion.
\end{proof}

The following example gives a generating set of minimum size for a direct power of the alternating group $A_4$. 
\begin{ex} \label{alternating-A4}
It is easy to see that $A_4$ is generated by the following two elements $$ \alpha=(1 ~2)(3~4) , ~\beta =(1~2~3).$$ Since $A_4$ is not perfect and $\alpha , \beta $ have coprime orders by Lemma \ref{rank-G^n}, the rank of $A_4^n$ is equal to $n$, for $n \geq 2$.
\end{ex}

\begin{theorem}\label{theorem1}
There exists a generating set of minimum size for $A_4^n$, for which the diameter of $A_4^n$ is at most $10 n $. 
\end{theorem}
\begin{proof}
As we mentioned before in Example \ref{alternating-A4}, the generating set $C$ constructed in the proof of Lemma \ref{rank-G^n} is a generating set of minimum size for $A_4^n$ for $n \geq 2$. We show that $\diam(A_4^n,C) \leq 10n$. Let $(g_1,g_2,\ldots,g_n) \in A_4^n$. By Remark \ref{simple}, it is enough to show that  $l_C(1,\ldots,1, g_i,1,\ldots,1) \leq 10$, for $1 \leq i \leq n$.
Because of the following equalities 
\begin{align*}
(1,\ldots,\overbrace{\alpha}^{i\,\mathrm{th}},\beta,\ldots,1)^3=(1,\ldots,\overbrace{\alpha}^{i\,\mathrm{th}},1,\ldots,1),\\
(1,\ldots,\alpha,\overbrace{\beta}^{i\,\mathrm{th}},\ldots,1)^4=(1,\ldots,1,\overbrace{\beta}^{i\,\mathrm{th}},\ldots,1),\\
(1,\ldots,\alpha,\overbrace{\beta}^{i\,\mathrm{th}},\ldots,1)^2=(1,\ldots,1,\overbrace{\beta^2}^{i\,\mathrm{th}},\ldots,1),
\end{align*}
we have
\begin{align*}
l_C(1,\ldots,\overbrace{\alpha}^{i\,\mathrm{th}},\ldots,1) \leq 3,\\
l_C(1,\ldots,\overbrace{\beta}^{i\,\mathrm{th}},\ldots,1) \leq 4,\\
l_C(1,\ldots,\overbrace{\beta^2}^{i\,\mathrm{th}},\ldots,1) \leq 2.
\end{align*}
On the other hand, the elements of $A_4$ can be represented over the generating set $\{\alpha, \beta \}$ as follows:
$$A_4=\{\alpha,\beta,\alpha^2,\alpha \beta,\beta \alpha, \beta^2, \alpha \beta \alpha, \alpha \beta^2, \beta \alpha \beta=\alpha \beta^2 \alpha, \beta^2 \alpha, \beta^2 \alpha \beta, \beta \alpha \beta^2 \}.$$
Now it is easy to see that the length of $(1,\ldots,\overbrace{g}^{i\,\mathrm{th}},\ldots,1)$ in the generating set $C$ is at most $10$ for every element $g \in A_4$, which completes the proof.
\end{proof}

\section{The diameter of $A_n ^2$, for $n\geq5$ }\label{sec:2}
Note that alternating groups $A_n$ for $n \geq 5 $ are  perfect. There is  a different approach to compute the rank of the direct power of perfect groups using the Eulerian function  of a  group (see \cite{Hall:1936,Wiegold:1974}). The following lemma is a consequence of the results in \cite{Hall:1936}.
\begin{lem}\label{Hall}
Let $G$ be a non-abelain simple group. If $G$ is generated by $n$ elements, then 
the set  $\{(a_{i1},a_{i2}\ldots,a_{ik}):i=1,\ldots,n\}$ will generate $G^k$ if and only if the following conditions are satisfied:
\begin{enumerate}
\item the set $\{a_{1i},a_{2i},\ldots,a_{ni}\}$ is a generating set of  $G$ for $i=1,\ldots,k$;
\item there is no automorphism $f: G\rightarrow G$ \\
which maps $(a_{1i},a_{2i},\ldots,a_{ni})$ to $(a_{1j},a_{2j},\ldots,a_{nj})$ for any $i\neq j$.
\end{enumerate}
\end{lem}
Furthermore, in \cite{Hall:1936} Hall shows that the alternating group $A_5$ satisfies Lemma \ref{Hall} with $n=2$ for $1 \leq k \leq 19$ and  not for $k \geq 20$.  

Therefore, the following is an immediate consequence of Lemma \ref{Hall}.
\begin{cor}\label{Hall2}
A pair $(s_1,\ldots,s_k), (t_1,\ldots,t_k)$ will generate $A_5^k$ if and only if the following conditions are satisfied:
\begin{enumerate}
\item the set $\{s_i,t_i\}$ is a generating set of  $A_5$ for $i=1,...,k$;
\item there is no automorphism $f: A_5\rightarrow A_5$ which maps $(s_i,t_i)$ to $(s_j,t_j)$ for any $i\neq j$.
\end{enumerate}
Furthermore, $k=19$ is the largest number for which these conditions can be satisfied.
That is,
the rank of $A_5^k$ is equal to $2$ if and only if $1 \leq k \leq 19$.
\end{cor}
%Now our goal is fo find generating sets of size two for $A_n^2$ for $n \geq 5$. First consider the following definition: 
Now we are ready to prove the following theorem. 
\begin{theorem}\label{theorem2}
Let $n \geq 5$. There exists a generating set of size two for $A_n^2$, for which the diameter of $A_n^2$ is at most $O(ne^{(c+1) (\log \,n)^4 \log \log n})$, for an absolute constant $c >0$.
\end{theorem}
\begin{proof}
For $n \geq 5$, let $a=(1\, 2 \,3\, \cdots\, n)$, $ b=(1 \,2)(3\, 4)$, $a'=(1\, 2\, 3\, \cdots\, n-1)$ and $ b'=(n-3 \,n-2)(n-1\, n)$. It is easy to see that $A=\{ a, b \}$ and $A'=\{a',b'\}$ are generating sets of $A_n$ for $n$ odd and $n$ even, respectively. Furthermore, if $n$ is odd, then $(a,b) , (b,a)$ are two non-equivalent 2-bases of $A_n$ and if $n$ is even, then $(a',b') , (b',a')$ are two non-equivalent 2-bases of $A_n$. By lemma \ref{Hall},  if $n$ is odd, then $A_n^2=<(a,b),(b,a)>$ and if $n$ is even, then $A_n^2=<(a',b'),(b',a')>$.  Let $G=\{(a,b),(b,a)\}$ and $G'=\{(a',b'),(b',a')\}$. Suppose for the moment that $n$ is odd. For $(x,y) \in A_n^2 $ we have $l_A(x,y)\leq l_A(x,1)+l_A(1,y)$. Combining this with the following equalities: 
\begin{align}
(a,b)^2=(a^2,1), 
(b,a)^n=(b,1),
(b,a)^2=(1,a^2), 
\text{and}
(a,b)^n=(1,b). 
\end{align}
 we obtain 
 \begin{equation}\label{1}
 \diam(A_n^2,G)\leq 2 \, n\, \diam(A_n, \{b,a^2\} ).
 \end{equation} 
 Replacing $\diam(A_n, \{b,a^2\} )$ with $O(e^{(c+1)(\log n)^4\log \log n})$ (see Theorem 6.6 in  \cite{Helfgott&Seress:2014}) in \ref{1} we get the desired conclusion. Similar arguments apply for the case that $n$ is even. 
\end{proof}

\begin{cor}\label{cor:An2}
 There exists a generating set of size two, for which the diameter of $A_n^2$ is at most $$2 (|A_n|-\rank(A_n))=n!-4,$$ for sufficiently large $n$.
\end{cor}
\begin{proof}
By theorem \ref{theorem2}, it is enough to show that $ne^{(c+1) (\log \,n)^4 \log \log n} \leq n! $, for sufficiently large $n$. Let $ \log n \geq (c+1)$, we have $n e^{(c+1) (\log \,n)^4 \log \log \,n} \leq n e^{(\log \,n)^6} $. Replacing $n!$ with the Stirling's formula, it is enough to say that 
 $e^{(\log \,n)^6+n} \leq n^{n-\frac{1}{2}} \sqrt{2\pi }$, for sufficiently large $n$, which is obvious.   
\end{proof}

\section{The diameter of a direct power of $A_5$}\label{sec:3}
We know that $19$ is the largest number for which the group  $A_5^k$ is generated by two elements for $1 \le k \le 19$ (see \cite{Hall:1936}).
Let $a=(1 2)(3 4),b=(1 2 3 4 5), c=(1 2 3),
d=(1 3 5),e=(2 4 5), f=(1 2 3 5 4),
g=(1 2 5 4 3),
h=(1 2 5 3 4), i=(1 3 2 5 4)$.
 We have checked with the Groups, Algorithms, Programming  (GAP) - a System for Computational Discrete Algebra- that  the pairs 
\begin{align*}
&( a, b),(b, a),( a, b^2 ),(b^2,a), (c,b),(b,c),(c,b^2),(b^2,c),(b,c^2),(c^2,b),(b^2,c^2),(c^2,b^2),\\
&(d,a),(a,d),(d,e),(b,f),(b,g),(b,h),(b,i) 
\end{align*}
 are 19 non-equivalent 2-basis of $A_5$.
%Note that the relation defined in the above definition is an equivalent relation. Let $S$ be a set of $n$-basis of $G$ which has exactly one element from each equivalent classes. The following proposition estates that $$max \{diam(G,A)| ~A \in S\}=max\{diam(G,A) |~|A|=n\}.$$ 
%
% \begin{prop}
%  Let $G$ be a finite group and $X=\{x_1,x_2,\ldots,x_n\}$ and $Y=\{y_1,y_2,\ldots,y_n\}$  two equivalent $n$-bases for $G$ (that is, there exists $\phi \in Aut(G)$ such that 
%  $y_i=x_i \phi$ for $i=1,2,\ldots,n$). Then $ l_X(g) = l_Y(\phi (g))$ for every $g\in G$, consequently $\diam(G,X)=\diam(G,Y)$. 
% \end{prop}
% 
% \begin{proof}
% Let $k$ be the smallest number for which  $g=x_{i_1}x_{i_2}\cdots x_{i_k}$, that is $l_X(g)=k$. We have $g \phi =(x_{i_1}\phi) (x_{i_2} \phi) \cdots (x_{i_k} \phi) = y_{i_1}y_{i_2}\cdots y_{i_k}$. It shows that $l_Y(g\phi) \leq k$. On the other hand, if $l_Y(g\phi) < k$
% then there exists $y_{t_1},y_{t_2},\ldots,y_{t_s} \in Y$ with $g \phi =y_{t_1} y_{t_2}\cdots y_{t_s}$ and $s < k$. Then $$g= (y_{t_1} \phi^{-1} )(y_{t_2}\phi^{-1})\cdots (y_{t_s}\phi^{-1})=x_{t_1}x_{t_2}\cdots x_{t_s},$$ that is  $l_X(g) \leq s$. Contradiction with $l_X(g)=k$. 
% \end{proof}
%
%
 By Corollary \ref{Hall2} we can build generating sets of size two for $A_5^k$, $1\leq k\leq 19$; but for proving theorem \ref{theorem3} we just need 8 of them. Let
 \begin{align*}
&C_1=\{a,b\},\\
&C_2=\{(a,b), (b,a)\},\\
&C_3=\{(a,b,a), (b,a,b^2)\},\\
&C_4=\{(a,b,a,b^2), (b,a,b^2,a)\},\\
&C_5=\{(a,b,a,b^2,c), (b,a,b^2,a,b)\},\\
&C_6=\{(a,b,a,b^2,c,b), (b,a,b^2,a,b,c)\},\\
&C_7=\{(a,b,a,b^2,c,b,c), (b,a,b^2,a,b,c,b^2)\},\\
&C_8=\{(a,b,a,b^2,c,b,c,b^2), (b,a,b^2,a,b,c,b^2,c)\},\\
%&C_9=\{(a,b,a,b^2,c,b,c,b^2,c^2), (b,a,b^2,a,b,c,b^2,c,b)\},\\
%&C_{10}=\{(a,b,a,b^2,c,b,c,b^2,b,c^2), (b,a,b^2,a,b,c,b^2,c,c^2,b)\},\\
%&C_{11}=\{(a,b,a,b^2,c,b,c,b^2,b,c^2,b^2), (b,a,b^2,a,b,c,b^2,c,c^2,b,c^2)\},\\
%&C_{12}=\{(a,b,a,b^2,c,b,c,b^2,b,c^2,b^2,c^2), (b,a,b^2,a,b,c,b^2,c,c^2,b,c^2,b^2)\},\\
%&C_{13}=\{(a,b,a,b^2,c,b,c,b^2,b,c^2,b^2,c^2,d), (b,a,b^2,a,b,c,b^2,c,c^2,b,c^2,b^2,a)\},\\
%&C_{14}=\{(a,b,a,b^2,c,b,c,b^2,b,c^2,b^2,c^2,d,a), (b,a,b^2,a,b,c,b^2,c,c^2,b,c^2,b^2,a,d)\},\\
%&C_{15}=\{(a,b,a,b^2,c,b,c,b^2,b,c^2,b^2,c^2,d,a,d), (b,a,b^2,a,b,c,b^2,c,c^2,b,c^2,b^2,a,d,e)\},\\
%&C_{16}=\{(a,b,a,b^2,c,b,c,b^2,b,c^2,b^2,c^2,d,a,d,b), (b,a,b^2,a,b,c,b^2,c,c^2,b,c^2,b^2,a,d,e,f)\},\\
%&C_{17}=\{(a,b,a,b^2,c,b,c,b^2,b,c^2,b^2,c^2,d,a,d,b,b), (b,a,b^2,a,b,c,b^2,c,c^2,b,c^2,b^2,a,d,e,f,g)\},\\
%&C_{18}=\{(a,b,a,b^2,c,b,c,b^2,b,c^2,b^2,c^2,d,a,d,b,b,b), (b,a,b^2,a,b,c,b^2,c,c^2,b,c^2,b^2,a,d,e,f,g,h)\},\\
%&C_{19}=\{(a,b,a,b^2,c,b,c,b^2,b,c^2,b^2,c^2,d,a,d,b,b,b,b), (b,a,b^2,a,b,c,b^2,c,c^2,b,c^2,b^2,a,d,e,f,g,h,i)\}.
\end{align*} 
Then for $1\leq n \leq 8$, the sets $C_n$ are generating sets of minimum size for groups $A_5^n$. 

\begin{theorem}\label{theorem3}
The diameter of $A_5^n$, for $1 \leq n \leq 8$, is at most $ 58\, n$.
\end{theorem}
\begin{proof}
Using GAP  \cite{Soicher:2022} we check that $\diam(A_5,C_1)=10$ and $\diam(A_5^2,C_2)=18$. Let $(x,y,z)$ be an arbitrary element in $A_5^3$. Then we have $$l_{C_3}(x,y,z)=l_{C_3}(x,1,z)+l_{C_3}(1,y,1).$$ On the other hand,  $(x,z) \in A_5^2=\langle (a, a), (b^2 , b^4) \rangle$ and  $\diam(A_5^2,\{(a,a),(b^2,b^4)\})=20$ and  $y\in A_5=\langle a,b^2\rangle$ and $\diam(A_5,\{a,b^2\})=9$. These facts together with the following equalities 
\begin{align*}
&(a,b,a)^2=(1,b^2,1),\\
& (b,a,b^2,)^5=(1,a,1),\\
&(a,b,a)^5=(a,1,a),\\
& (b,a,b^2,)^2=(b^2,1,b^4),
\end{align*}
lead to
 \begin{align*}
\diam(A_5^3,C_3)\leq 5\times 9+5\times 20=5 \times (9 +  20)=145.
\end{align*} 

In the same manner we can see that

\begin{align*}
&\diam(A_5^4,C_4)\leq 5 \times (20 + 20)= 200,\\
&\diam(A_5^5,C_5)\leq 6 \times (20 + 25)= 270,\\ 
&\diam(A_5^6,C_6)\leq 6 \times (25 +  25)= 300,\\ 
&\diam(A_5^7,C_7)\leq 6 \times (25 + 30)=330,\\
&\diam(A_5^8,C_8)\leq 6 \times (30 + 30)=360.
\end{align*}

\end{proof}

Forasmuch as by increasing $n$ the Cayley Graph of $A_5^n$ is growing exponentially, we could not calculate the diameter of $A_5^n$ for $n \geq 4$ with the GAP, hence we could not estimate the diameter of $A_5^n$ for $n \geq 9$ with the technique which is used in the proof of theorem \ref{theorem3}.

In 2015, the second author conjectured that the diameter of $G^n$ is growing polynomially with respect to $n$. More precisely, she conjectured that if $G$ is a finte group, then diameter of $G^n$ is at most $n (|G|-\rank(G)) $, which is called the strong conjecture. The strong conjecture has been proved for abelian groups in \cite{Karimi:2017}. Another version of the strong conjecture called the weak conjecture states that if $G$ is a finite group, then there exists a generating set of minimum size for $G^n$, for which the diameter of $G^n$ is at most $n (|G|-\rank(G)) $. The weak conjecture is proved for nilpotent groups, dihedral groups and some power of imperfect groups in \cite{Karimi:2017}. Recently, it is shown that for a solvable group $G$, the diameter of $G^n$ grows polynomially with respect to $n$ \cite{Azad&Karimi:2022}. In this paper,  theorems \ref{theorem1}, \ref{theorem3} and corollary  \ref{cor:An2} are more pieces of evidence for the weak conjecture.

\section{Acknowledgments}
 The second author wishes to thank Arak University, for the invitation and hospitality, and the International Science and Technology Interactions (ISTI) for financial support.

\bibliographystyle{amsplain}
\bibliography{nasimref}

\end{document}